\documentclass[11pt]{amsart}
\usepackage{amsmath,amsthm,amssymb, mathrsfs}
\usepackage[tmargin=1.0in,bmargin=1.0in,rmargin=1in,lmargin=1in]{geometry}
\usepackage[breaklinks=true]{hyperref}

\newcommand{\N}{\ensuremath{\mathbb N}}
\newcommand{\Z}{\ensuremath{\mathbb Z}}
\newcommand{\R}{\ensuremath{\mathbb R}}
\newcommand{\Id}{\ensuremath{I}}
\DeclareMathOperator{\diag}{\ensuremath{diag}}

\theoremstyle{plain}

\newtheorem{theorem}[equation]{Theorem}
\newtheorem{corollary}[equation]{Corollary}
\newtheorem{proposition}[equation]{Proposition}
\newtheorem{lemma}[equation]{Lemma}
\newtheorem{conjecture}[equation]{Conjecture}

\theoremstyle{definition}

\newtheorem{definition}[equation]{Definition}

\newtheorem{remark}[equation]{Remark}

\theoremstyle{remark}

\numberwithin{equation}{section}

\begin{document}
\title{The critical exponent conjecture for powers of
doubly nonnegative matrices}
\author[Dominique Guillot \and Apoorva Khare \and Bala
Rajaratnam]{Dominique Guillot \and Apoorva Khare \and Bala
Rajaratnam\\Stanford University}
\address{Departments of Mathematics and Statistics, Stanford University}
\thanks{Partial support for this work came for D.G. from an NSERC
Postdoctoral Fellowship (Canada); for A.K. from the DARPA Grant YFA
N66001-11-1-4131; and for B.R. from grants NSF DMS-1106642 and DARPA-YFA
N66001-11-1-4131.}
\date{\today}

\begin{abstract}
Doubly non-negative matrices arise naturally in many setting including Markov random fields (positively banded graphical models) and in the convergence analysis of Markov chains. In this short note, we settle a recent conjecture by C.R.~Johnson et al.~[Linear Algebra Appl.~435 (2011)] by proving that the critical exponent
beyond which all continuous conventional powers of $n$-by-$n$ doubly
nonnegative matrices are doubly nonnegative is exactly $n-2$. We show
that the conjecture follows immediately by applying a general
characterization from the literature. We prove a stronger form of the
conjecture by classifying all powers preserving doubly nonnegative
matrices, and proceed to generalize the conjecture for broad classes of functions.
We also provide different approaches for settling the original
conjecture.
\end{abstract}
\maketitle

\section{Introduction and main result}
The study of operations preserving different notions of positivity is an important topic in 
matrix analysis (see e.g.~Bhatia \cite{bhatiaPDM}, Horn and Johnson \cite{Horn_and_Johnson_Topics}). 
The purpose of this short note is to prove the \emph{critical exponent
conjecture} for doubly nonnegative matrices formulated recently by
Johnson {\it et al.} \cite{johnson_et_al_2011}. A real symmetric matrix
$A$ is said to be \emph{doubly nonnegative} if it is both positive
semidefinite, and  entrywise nonnegative. Given a real function $f: [0, \infty) \rightarrow [0, \infty)$ and a doubly nonnegative matrix $A$, we define the matrix $f(A)$ using the
spectral decomposition of $A$. Namely, if $A = U D U^T$, where $U$ is
orthogonal and $D = \diag(d_{11}, \dots, d_{nn})$ is diagonal, define
$f(A) := U f(D) U^T$, where $f(D) := \diag(f(d_{11}), \dots, f(d_{nn}))$.
Assuming $A$ is doubly nonnegative, it is very natural to seek conditions
under which $f(A)$ is also doubly nonnegative.

A similar question can be asked when the function $f$ is applied
entrywise to the elements of $A$. Denote by $f[A] := (f(a_{ij}))$. In the
particular case where $f(x) = x^\alpha$ for $0 \leq \alpha \in \R$, it has been shown by FitzGerald
and Horn \cite{FitzHorn} (see also Horn and Johnson \cite[Chapter
6.3]{Horn_and_Johnson_Topics}) that $f[A]$ is doubly nonnegative for
every doubly nonnegative matrix $A$ if and only if $\alpha \in \N \cup [n-2, \infty)$. Note that in each of the two questions above, one of the two
nonnegativity conditions is merely used to ensure that $f(A)$ or $f[A]$ is
well-defined, and the real question is whether the other notion of
nonnegativity is preserved.

Johnson {\it et al.}~are interested in the first question above for the
power functions $f(x) = x^\alpha$ for real $\alpha$. Namely, which conventional powers
preserve the set of doubly nonnegative matrices of a given order? This problem was also considered by Audenaert \cite{Audenaert}. In \cite{johnson_et_al_2011}, the
following phase transition was shown to occur for every integer $n>0$.

\begin{theorem}[{\cite[Theorem 2.1]{johnson_et_al_2011}}]
There is a function $\mu: \N \rightarrow [0, \infty)$ such that for any $n$-by-$n$ doubly
nonnegative matrix $A$, the matrix $f(A) = A^\alpha$ is doubly
nonnegative for $\alpha \geq \mu(n)$.
\end{theorem}

The smallest such $\mu(n)$ always exists by continuity and is denoted by
$m(n)$. The quantity $m(n)$ is called the (conventional) \emph{critical
exponent} (see \cite{johnson_et_al_2011}), and it is of great interest to
identify it. Johnson et al.~\cite{johnson_et_al_2011} provide a lower
bound for $m(n)$ for all $n$, and a sharp upper bound for small values of
$n$.

\begin{theorem}[{\cite[Theorems 3.1, 3.2,
4.1]{johnson_et_al_2011}}]\label{Tjohnson}
The critical exponent $m(n) \geq n-2$. If $n<6$, then the critical
exponent $m(n) = n-2$.
\end{theorem}

\noindent The authors then proceed to conjecture if the above result is
true in general:

\begin{conjecture}[{\cite[Johnson, Lins, and Walch,
2011]{johnson_et_al_2011}}]\label{Ccritexp}
For all integers $n \geq 2$, the critical exponent $m(n) = n-2$. 
\end{conjecture}

The goal of this note is to settle this conjecture and prove that indeed
$m(n) = n-2$ for all $1 < n \in \N$. More generally, we prove this
conjecture (in fact, a stronger statement of it) for a larger class of
functions. In order to state our main result, we first need some
notation.

\begin{definition}
Let $n \geq 2$ be an integer. Denote the set of doubly nonnegative $n \times n$ matrices by $DN_n$. Now
given a function $f : [0,\infty) \to [0,\infty)$, define the {\em
critical exponent} for $f$ to be
\begin{equation}
m(f,n) := \inf\ \{0 \leq t \in \R \ : \ A^\alpha f(A) \in DN_n \ \forall A \in
DN_n,\forall \alpha \geq t \}.
\end{equation}
\end{definition}

We now state the main result of this paper.

\begin{theorem}[Critical exponent conjecture, stronger form]\label{Tmain}
Given $u>0$ and $\beta \in \R$, the critical exponent for $f(x) :=
(x+u)^{-\beta}$ is $m(f,n) = n-2+\max(\beta,0)$. Moreover, below the critical exponent there are only
finitely many values $0 \leq \alpha < m(f,n)$ such that the map $g(x)
= x^\alpha (x+u)^{-\beta}$ preserves $DN_n$, and these are contained in
the set 
\begin{equation}\label{Eexcept}
\{m +\max(\beta,0) : m=0, 1, \dots, n-3\}.
\end{equation} 
\end{theorem}

\begin{remark}
Note that the critical exponent conjecture \ref{Ccritexp} for powers by Johnson et al.~\cite{johnson_et_al_2011} is the special case where $\beta = 0$.
Additionally, Theorem \ref{Tmain} strengthens the statement
of the critical exponent conjecture significantly in two ways: (1) it
discusses the behavior of $A^\alpha f(A)$ for values of $\alpha$ lower
than the critical exponent, and (2) it considers more general functions
than conventional matrix powers.
\end{remark}

We will also see in Section \ref{S2} that showing that $g(x) =
x^{n-1}(x+u)^{-1}$ preserves $DN_n$ for all $u > 0$ actually implies the critical
exponent conjecture. This reason also motivates our trying to prove Theorem
\ref{Tmain} which is a more general result than the standard form of the critical exponent conjecture.

\section{Proof of the Strong critical exponent conjecture}

We now settle Conjecture \ref{Ccritexp} by proving the more
general result given by Theorem \ref{Tmain}. To do so, we shall employ
the following result from the literature.

\begin{theorem}[{\cite[Micchelli and Willoughby, Corollary
3.1]{micchelli_et_al_1979}}]\label{Tmicc}
Suppose $g(x) \geq 0$ is continuous for $x \geq 0$. Then $g$ preserves
$DN_n$ if and only if all the divided differences of $g$ in $[0,\infty)$
of order $1, 2, \dots, n-1$ are nonnegative.
\end{theorem}

\noindent Theorem \ref{Tmicc} is shown by replacing $g$ by the Newton
interpolation polynomial taking values $g(\lambda_i)$ at the eigenvalues
$0 \leq \lambda_1 \leq \dots \leq \lambda_n$ of $A$, and proving that each
summand is doubly nonnegative, i.e., 
\begin{equation}
(A-\lambda_1 \Id)\cdots (A-\lambda_k \Id) \in DN_n \qquad \forall A \in
DN_n,\ 1 \leq k \leq n-1.
\end{equation}

Note that if $g$ is $C^{n-1}$, then $g$ preserves $DN_n$ if and only
if $g^{(i)}(x) \geq 0$ for all $0 \leq i < n$ and $x > 0$ (by the mean
value theorem for divided differences). As a consequence, we have the following corollary. 

\begin{corollary}\label{cor:strong_conj}
Conjecture \ref{Ccritexp} holds for all integers $n \geq 2$. Moreover, $g(x) = x^\alpha$ preserves $DN_n$ if and only if $\alpha \in \N \cup [n-2, \infty)$. 
\end{corollary}
\begin{proof}
Immediate by considering $g(x) = x^\alpha$. 
\end{proof}

\noindent Note that the set of powers which preserves $DN_n$ for conventional powers (Corollary \ref{cor:strong_conj}) is exactly the same as for Hadamard powers \cite{FitzHorn}: $\alpha \in \N \cup [n-2,\infty)$.

\begin{remark}
Besides settling the critical exponent conjecture for powers, Theorem \ref{Tmicc} can also be used to establish the existence of a critical exponent in very general settings. For instance, a family of smooth
functions $(f_\alpha)_{\alpha \geq 0}$ will have a ``critical index''
$m(n)$ for $DN_n$, if and only if 
\begin{equation}\label{eqn:pos_der}
f_\alpha^{(i)}(x) \geq 0, \qquad \forall x > 0,\ i=0, 1, \dots, n-1
\end{equation}

\noindent for all $\alpha$ large enough. A stronger statement would be
that \eqref{eqn:pos_der} holds if and only if $\alpha \in S(n) \cup [m(n), \infty)$ for some finite set $S(n)$. We show in Theorems
\ref{Tmain} and \ref{Texp} that this is indeed the case for large
families of functions.
\end{remark}

To prove Theorem \ref{Tmain} using Theorem \ref{Tmicc}, we need the
following preliminary result. 

\begin{lemma}\label{lem:derivatives}
For all $\alpha, \beta \in \R$,  $u, x > 0$, and $0 \leq r \in \Z$,
\begin{align*}
\frac{d^r}{d x^r} \frac{x^\alpha}{(x+u)^\beta} = &\ \frac{x^{\alpha -
r}}{(x+u)^{\beta + r}} \sum_{i=0}^r \binom{r}{i} u^i x^{r-i}
\prod_{j=0}^{i-1} (\alpha -j) \prod_{j=i}^{r-1} (\alpha -\beta -j)\\
= &\ \frac{x^{\alpha - r}}{(x+u)^{\beta + r}} \sum_{i=0}^r \binom{r}{i}
u^i x^{r-i} \frac{\Gamma(\alpha + 1) \Gamma(\alpha - \beta -
i+1)}{\Gamma(\alpha-i+1) \Gamma(\alpha-\beta-r+1)}.
\end{align*}
\end{lemma}
\begin{proof}
The proof is by induction on $r \geq 0$; the base case is clear. Now
given the result for $r$, differentiate the right-hand side using the
quotient rule. A typical summand involves differentiating
$\frac{x^{\alpha - i}}{(x+u)^{\beta + r}}$; the derivative is
\begin{align*}
&\ \frac{(x+u)^{\beta + r}(\alpha - i) x^{\alpha - i - 1} - (\beta + r)
(x+u)^{\beta + r-1} x^{\alpha - i}}{(x+u)^{2(\beta + r)}}\\
= &\ \frac{x^{\alpha - i - 1}}{(x+u)^{\beta + r + 1}} \left( (\alpha -
i)(x+u) - (\beta + r)x \right)\\
= &\ \frac{1}{(x+u)^{\beta + r + 1}} \left( (\alpha - \beta - r -
i)x^{\alpha - i} + (\alpha - i) u x^{\alpha - i - 1} \right).
\end{align*}

\noindent Now multiply this quantity by the coefficient $\binom{r}{i} u^i
\Gamma_{r,i}$, where
\[ \Gamma_{r,i} := \frac{\Gamma(\alpha + 1) \Gamma(\alpha - \beta -
i+1)}{\Gamma(\alpha-i+1) \Gamma(\alpha-\beta-r+1)}. \]

\noindent Add the products from $i=0$ to $r$, and collect like powers of
$u$ to get the $(r+1)$th derivative. We claim that this is indeed of the
claimed form, which would complete the induction step and show the
result. But this is a straightforward verification for each term: the
coefficient of $u^0$ in the $(r+1)$th derivative comes from only the
$i=0$ term:
\[ \frac{x^\alpha (\alpha - \beta - r)}{(x+u)^{\beta + r + 1}}
\Gamma_{r,0} = \frac{x^\alpha}{(x+u)^{\beta + r + 1}} \Gamma_{r+1,0} \]

\noindent as desired. Similarly, the coefficient of $u^{r+1}$ also comes
from only the $i=r$ term:
\[ \frac{x^{\alpha-r-1} (\alpha - r)}{(x+u)^{\beta + r + 1}}
\Gamma_{r,r} = \frac{x^{\alpha-r-1}}{(x+u)^{\beta + r + 1}}
\Gamma_{r+1,r+1}. \]

\noindent It remains to show that the coefficient of $\displaystyle
\frac{x^{\alpha - i} u^i}{(x+u)^{\beta + r + 1}}$ (for $1 \leq i \leq r$)
in the sum in the $(r+1)$th derivative, which equals $\binom{r+1}{i}
\Gamma_{r+1,i}$, comes from adding up the two terms in the above sum of
derivatives. In other words, it suffices to show that
\[ \binom{r}{i} (\alpha - \beta - r - i) \Gamma_{r,i} +
\binom{r}{i-1} (\alpha - i + 1) \Gamma_{r,i-1} = \binom{r+1}{i}
\Gamma_{r+1,i}. \]

\noindent Now note that each of the three terms $\binom{r}{i-1}
\Gamma_{r,i-1}, \binom{r}{i} \Gamma_{r,i}, \binom{r+1}{i} \Gamma_{r+1,i}$
have the following expression common:
\[ \frac{r!}{i! (r-i+1)!} \prod_{j=0}^{i-2} (\alpha - j)
\prod_{j=i}^{r-1} (\alpha - \beta - j). \]

\noindent Thus, it suffices to show the equality after having
``cancelled" this common expression from all terms - namely, to show that
\begin{align*}
&\ (r-i+1) (\alpha - \beta - r - i) (\alpha - i + 1) + i (\alpha - i + 1)
(\alpha - \beta - i + 1)\\
= &\ (r+1) (\alpha - i + 1)(\alpha - \beta - r).
\end{align*}

\noindent Once again, it suffices to show this without the common
expression $(\alpha - i + 1)$ in each of the terms. But this is a
straightforward computation.
\end{proof}

Finally, we now prove the main result of this note. 

\begin{proof}[Proof of Theorem \ref{Tmain}]
Suppose first that $\alpha \geq n-2 + \max(\beta,0)$. Then for all $0
\leq i < n$ and $x > 0$, the $i$th derivative
\[ \frac{d^i}{d x^i} \frac{x^\alpha}{(x+u)^\beta} \]

\noindent is indeed defined and can be computed using Lemma
\ref{lem:derivatives}. It is easy to check that all of these
derivatives are nonnegative, since $x,u > 0$, and $\alpha - (n-2)$ and $\alpha - \beta - (n-2)$ are nonnegative. By Theorem \ref{Tmicc} (see the remarks preceding Corollary \ref{cor:strong_conj}), we infer that $x^\alpha(x+u)^{-\beta}$ preserves $DN_n$
for all $\alpha \geq n-2 + \max(\beta,0)$.

To complete the proof, suppose now that $0 \leq \alpha < n-2 +
\max(\beta,0)$ and $\alpha$ is not in the set \eqref{Eexcept}.
Define $g(x) = x^\alpha (x+u)^{-\beta} : [0,\infty) \to [0,\infty)$.
By Theorem \ref{Tmicc}, it suffices to show 
that there is at least one $0 \leq i < n$ and $x > 0$ such that $g^{(i)}(x) < 0$. There are two cases:
\begin{enumerate}
\item Suppose $\beta < 0$. Now compute by Lemma \ref{lem:derivatives},
\[ g^{(m+2)}(x) = \frac{x^{\alpha - m - 2}}{(x+u)^{\beta+m+2}} \left[
u^{m+2} \prod_{j=0}^{m+1} (\alpha - j) + x \sum_{i<m+2} \binom{m+2}{i} u^i x^{m+1-i}
\prod_{j=0}^{i-1} (\alpha-j) \prod_{j=i}^{m+1} (\alpha-\beta-j) \right].
\]

\noindent Now if $\alpha \in (m,m+1)$ for some integer $m \in [0,n-3]$,
then $u^{m+2} \prod_{j=0}^{m+1} (\alpha - j)$ is the smallest degree term in $x$ and 
is negative. Therefore $g^{(m+2)}(x) < 0$ for small $x>0$. 

\item Suppose $\beta \geq 0$. Once again using Lemma
\ref{lem:derivatives},
\[ g^{(m+2)}(x) = \frac{x^{\alpha - m - 2}}{(x+u)^{\beta+m+2}} \left[
x^{m+2} \prod_{j=0}^{m+1} (\alpha - \beta - j) + x \sum_{i>0} \binom{m+2}{i} u^i
x^{m+1-i} \prod_{j=0}^{i-1} (\alpha-j) \prod_{j=i}^{m+1} (\alpha-\beta-j)
\right]. \]

\noindent Now if $\alpha - \beta \in (m,m+1)$ for
some integer $m \in [0,n-3]$, then a similar analysis as in the previous
case, but this time with $x \to \infty$, implies that $g^{(m+2)}(x) < 0$
for some $x>0$.
\end{enumerate}
\end{proof}

\subsection{Alternate proof}\label{S2}

We now provide an alternate proof of the original critical exponent
conjecture in \cite{johnson_et_al_2011}. Similar to Johnson et al.'s result in \cite{johnson_et_al_2011} (see Theorem \ref{Tjohnson}), we only show the conjecture (i.e., that $m(f,n) = n-2$ for $f(x) \equiv
1$) for small values of $n$, although we also indicate an
approach for all $n$.

The first simplification is in showing that to verify whether or not a
given function $f$ preserves $DN_n$, it is enough to focus on a
particular off-diagonal entry:

\begin{lemma}\label{Lperm}
Given $f : [0,\infty) \to [0,\infty)$ and a fixed pair of integers $1 \leq i \neq j \leq
n$, the matrix $f(A)$ is doubly nonnegative for all $A \in DN_n$ if and only if 
$(f(A))_{ij} \geq 0$ for each $A \in DN_n$.
\end{lemma}

\begin{proof}
One implication is immediate from the definition. Conversely, given any $A = U D U^T$ with $U$ orthogonal and $D$ diagonal, note that $f(A) = U f(D) U^T$ is positive semidefinite. Hence its diagonal entries are nonnegative. We now consider the off-diagonal entries of $f(A)$. Note that for every $n$-by-$n$ permutation matrix $P$, 
\[ f(P A P^T) = f((PU) D (PU)^T) = (PU) f(D) (PU)^T = P f(A) P^T. \]

\noindent Since conjugating by $P$ preserves $DN_n$, hence $(P f(A)
P^T)_{ij} \geq 0$ for all $P,A$. This is equivalent to saying that
$f(A)_{i'j'} \geq 0$ for all $i' \neq j'$. 
\end{proof}

Next, note by Theorem \ref{Tmicc} and Lemma \ref{lem:derivatives}, that for every $n \in \N$,
\begin{equation}\label{Efact}
A^{n-1}(A+u \Id)^{-1} \in DN_n, \qquad \forall A \in DN_n, u>0.
\end{equation}

\noindent We now show that this fact alone implies the critical exponent
conjecture \ref{Ccritexp}.

\begin{proposition}\label{Preduction}
Fix an integer $n = n_0 \geq 2$. Then Equation \eqref{Efact} for $n = n_0$ implies the critical exponent
conjecture \ref{Ccritexp} for $n = n_0$.
\end{proposition}

\begin{proof}
By \cite[Theorem 3.2]{johnson_et_al_2011}, we know that $m(n_0) \geq n_0-2$.
To show that $m(n_0) = n_0-2$, assume first that $A \in
DN_{n_0}$ is nonsingular. Clearly, $A^k$ is doubly nonnegative if $A \in DN_{n_0}$ and $k
\in \N$. Therefore, it suffices to prove that $A^q$ is doubly nonnegative
for every $k < q < k+1$ and every integer $k \geq n_0-2$. To show this,
note that for $k < q < k+1$ and $x>0$, the following formula holds:
\begin{equation}\label{eqn:contour_int}
x^q = \frac{\sin((q-k)\pi)}{\pi} \int_0^\infty t^{q-k-1} x^{k+1}
(x+t)^{-1}\ dt.
\end{equation}

\noindent Equation \eqref{eqn:contour_int} is shown by a standard contour integration followed by an
application of the residues theorem (see e.g.~\cite[Chapter V, Example
2.12]{conway_complex}). But now the matrix $A^q = U D^q U^T$ can be
written as:
\begin{equation}\label{Ebhatia}
A^q = \frac{\sin((q-k)\pi)}{\pi} \int_0^\infty t^{q-k-1} A^{k+1} (A+t
\Id)^{-1}\ dt
= \frac{\sin((q-k)\pi)}{\pi} \int_0^\infty t^{q-k-1} f_{k,t}(A)\ dt,
\end{equation} 

\noindent where $f_{k,t}(x) := x^{k+1} (x+t)^{-1}$ for $t \geq 0$ and $x
> 0$. Note that $\sin((q-k)\pi) > 0$ since $0 < q-k < 1$. Moreover, since
for every $A \in DN_{n_0}$ and $k \geq n_0-2$, the matrix $f_{k,t}(A)$ is
doubly nonnegative by Equation \eqref{Efact}, it follows immediately that
$A^q$ is doubly nonnegative. 

Finally, if $A$ is singular, then $A^q = \lim_{u \rightarrow 0^+} (A + u\Id)^q$ and the result follows. 
\end{proof}

\begin{remark}
Equation \eqref{Ebhatia} is frequently used in the study of positive
definite matrices. See e.g.~Bhatia \cite[Theorem
1.5.8]{bhatiaPDM}.
\end{remark}

Given Proposition \ref{Preduction}, it is thus of interest to prove Equation \eqref{Efact} from first
principles in order to provide an elementary proof of the 
critical exponent conjecture in \cite{johnson_et_al_2011}.
We conclude this section with a case-by-case verification of Equation \eqref{Efact} for
$n \leq 3$. Such an approach (but using different arguments) was adopted by Johnson {\it et
al.}~in \cite{johnson_et_al_2011} when they showed the conjecture for
$n<6$.

\begin{proposition}
The critical exponent conjecture \ref{Ccritexp} holds for $n \leq 3$.
\end{proposition}

\begin{proof}
In light of Proposition \ref{Preduction}, it suffices to show that
Equation \eqref{Efact} holds (for $n \leq 3$). First note that $A^{n-1} (A + u
\Id)^{-1}$ is clearly positive semidefinite. By using Lemma \ref{Lperm}, it now suffices to show that the
$(1,2)$-entry of $A^{n-1} (A + u \Id)^{-1}$ is nonnegative. Denote this entry by $a_{12,n}$. It is then easily shown that for $A \in DN_2$,
\[ a_{12,2} = \frac{a_{12} u}{\det (A + u \Id_2)} \geq 0. \]

\noindent Similarly, for $n=3$,
\[ a_{12,3} = \det(A + u \Id_3)^{-1} \left( a_{12} \det A + a_{12} u
(C_{11} + C_{22} + C_{33}) + u^2(a_{12} a_{11} + a_{12} a_{22} + a_{13}
a_{23}) \right), \]

\noindent where $C_{ii}$ are the principal cofactors/minors of $A$, hence
are nonnegative. This concludes the proof since all entries of $A$ are
nonnegative.
\end{proof}

\section{Generalization of the strong critical exponent conjecture}

Note that Theorem \ref{Tmicc} can be
used to systematically study and resolve the (strong) critical exponent
conjecture for broad classes of functions. We conclude this note by showing such a
result.

\begin{theorem}\label{Texp}
Let $f: [0, \infty) \rightarrow [0, \infty)$ be a $C^{n-1}$ function satisfying
\begin{equation}\label{Epositive}
f(x) > 0, \quad \lim_{x \to 0^+} f(x) > 0, \quad f^{(i)}(x) \geq 0,
\quad \lim_{x \to 0^+} f^{(i)}(x) = M_i < \infty, \quad\forall x>0, \ 0 < i < n. 
\end{equation}
Then the values $0 \leq \alpha \in \R$ such that $x^\alpha f(x)$
preserves $DN_n$ are $\N \cup [n-2,\infty)$.
In particular, the critical exponent for all such functions $f$ is
$m(f,n) = n-2$.
\end{theorem}
\begin{proof}
Set $g(x) := x^\alpha f(x)$. In order to use Theorem \ref{Tmicc}, we evaluate $g^{(i)}(x)$ using the product rule:
\[ g^{(i)}(x) = \sum_{l=0}^i \binom{i}{l} \prod_{j=0}^{l-1} (\alpha -j)
x^{\alpha -l} f^{(i-l)}(x). \]

\noindent If $\alpha \in \N$ then every product $\prod_{j=0}^{l-1}
(\alpha-j)$ that includes a negative term also includes $0$, whence all
derivatives of $g$ are nonnegative at $x>0$.
Similarly, if $\alpha \geq n-2$, then $g^{(i)}(x) \geq 0$ for $x > 0$.
Finally if $\alpha \notin \N$ and $\alpha < n-2$, assume $\alpha \in
(m,m+1)$ for some integer $m \in [0, n-3]$. Then the $l=m+2$ term
for $g^{(m+2)}(x)$ is negative, and hence $g^{(m+2)}(x) < 0$ for
small values of $x > 0$ by \eqref{Epositive}.
\end{proof}

\begin{remark}
Note that the critical exponent conjecture in
\cite{johnson_et_al_2011} is again a special case of Theorem \ref{Texp} - with $f(x) \equiv 1$.
Moreover, the cone $\mathscr{F}$ of functions $f$ satisfying
\eqref{Epositive} is a very large family that contains all absolutely
monotonic functions with nonzero constant term as well as
exponential-type functions such as $e^x$ and $\cosh(x)$. Additionally,
$\mathscr{F}$ is closed under multiplication, exponentiation,
composition, and taking nontrivial nonnegative linear combinations.
\end{remark}
\bibliographystyle{plain}
\bibliography{biblio}
\end{document}